\documentclass[11pt]{article}

\usepackage{amssymb}
\usepackage{amsmath}
\usepackage{amsthm}
\usepackage{xcolor}

\begin{document}

\parskip 1ex            % sets the gap between paragraphs

\parindent 5ex		% an ex is the width of an ex

\newcommand\mb{\mathbb}

\newcommand\mc{\mathcal}

\newcommand{\pbt}[1]{\textcolor{blue}{#1}}
\newcommand{\pbc}[1]{\textcolor{blue}{[{\bf PB}: {\em #1}]}}

\newtheorem{theorem}{Theorem}[section]

\newtheorem{lemma}[theorem]{Lemma}

\newtheorem{proposition}[theorem]{Proposition}

\newtheorem{corollary}[theorem]{Corollary}

\newtheorem{claim}[theorem]{Claim}

\newtheorem{conjecture}[theorem]{Conjecture}

\title{A natural barrier in random greedy hypergraph matching}

\author{
Patrick Bennett\thanks{Mathematics Department, Western Michigan 
University, Kalamazoo, MI 49008, USA. Email: {\tt patrick.bennett@wmich.edu}.
Research supported in part by NSF grant DMS-1001638 and Simons Foundation grant \#426894.}
\and
Tom Bohman\thanks{Department of Mathematical Sciences, Carnegie Mellon
University, Pittsburgh, PA 15213, USA. Email: {\tt tbohman@math.cmu.edu}.
Research supported in part by NSF grants DMS-1001638 and DMS-1100215.}
}
\date{}

\maketitle

\begin{abstract}

Let $r \ge 2$ be a fixed constant and let
$ {\mathcal H}$ be an $r$-uniform,
$D$-regular hypergraph on $N$ vertices.
Assume further that $ D \to \infty$ as $N \to \infty$
and that
degrees of pairs of vertices in ${\mathcal H}$ are
at most $L$ where $L \ = D/ (\log N)^{\omega(1)}$.
We consider the random greedy algorithm for forming a matching
in $ \mc{H}$.  We choose a matching at random
by iteratively
choosing edges uniformly at random to be in the matching and
deleting all edges that share at least one vertex with a
chosen edge before moving on to the next choice.
This process terminates when there are no edges remaining
in the graph.  We show that with high probability the
proportion of vertices of $ {\mathcal H}$ that are not
saturated by the final matching is at most
$ (L/D)^{ \frac{ 1}{ 2(r-1) } + o(1) } $.  This point is
a natural barrier in the analysis of the random greedy hypergraph
matching process.
\end{abstract}

\section{Introduction}

Let $r \ge 2 $ be a fixed constant and let
$ {\mathcal H}$ be an $r$-uniform,
$D$-regular hypergraph on vertex set $V$ where  $|V|=N$ and
$ D \to \infty$ as $N \to \infty$.  We study the
evolution of the random greedy matching algorithm on ${\mathcal H}$.
This process forms a matching (i.e. a collection of pairwise disjoint
edges) in $ {\mathcal H}$ by making a series of random choices.  We begin
with $\mc{M}(0) = \emptyset$, $ \mc{H}(0) = \mc{H}$ and $V(0) = V$.  In iteration $i$ an edge $E_i$ is chosen
uniformly at random from $ \mc{H}(i-1) $ and added to $ \mc{M}(i-1) $ to form the matching
$ \mc{M}(i) $.  We then form $ \mc{H}(i) $ by setting $V(i) = V(i-1) \setminus E_i$ and deleting from $ \mc{H}(i-1) $
all edges that intersect $ E_i$.  The process proceeds until the step $M$ where
$ {\mathcal H}(M) $ is
empty.  We are interested in the likely value
of $M$; that is, we are interested in the number of
edges in the matching produced by the random greedy process.

The random greedy packing algorithm for producing a partial Steiner system
is an important special case of this process.  Let $ 1 < \ell < k $ be
fixed integers.  Define $ \mc{H}_{\ell,k} $ to
be the hypergraph on vertex set $ \binom{[n]}{\ell} $ with edge set consisting
of all sets of the form $ \binom{A}{\ell} $ where $ A \in \binom{[n]}{k} $.  Note
that a matching in $ {\mathcal H}_{\ell,k} $ corresponds to a collection
of $k$-element subsets of $[n]$ with the property that the intersect of any
pair of sets in the collection has cardinality less than $ \ell$; that is,
a matching in $ {\mathcal H}_{\ell,k} $ gives a partial $ (n,k,\ell)$-Steiner system.
The random greedy matching algorithm applied to $ {\mathcal H}_{\ell,k} $ is also
known as random greedy packing.  This process is related to the celebrated
R\"odl nibble \cite{rodl}, which is a semi-random variation
on random greedy packing.  The R\"odl nibble was introduced in the solution of
the Erd\H{o}s and Hanani conjecture \cite{erdos}, which states that
for every fixed $\ell,k$ there is a matching in $ {\mathcal H}_{\ell,k} $ that
saturates $ (1-o(1)) \binom{n}{\ell} $ vertices.

In this paper
we study the general random greedy matching algorithm
by establishing dynamic concentration of the number of edges and
the vertex degrees in the remaining hypergraph $ \mc{H}(i) $.  Let $ Q(i) $ be the
number of edges in $ \mc{H}(i) $ and let $ d_v(i) $ be the degree of vertex $v$
in $ \mc{H}(i)$.  We aim to show that $ Q(i) $ and $d_v(i)$, appropriately
scaled, are tightly concentrated around expected trajectories that we express
as smooth functions on the reals.  In order to describe the trajectories we introduce
a continuous time $t$ which we relate to the steps of the process by setting
\[ t = t(i) = \frac{i}{N}. \]
Our study is guided by the following probabilistic intuition:
we suspect that $\mc{H}(i)$ resembles a subhypergraph of $ \mc{H} $ chosen
uniformly at random from the collection of all subhypergraphs induced by
$ N- ir$ vertices.  So we anticipate that $ \mc{H}(i)$ resembles
a subhypergraph of $ \mc{H}$ induced by a random subset of the vertices where
each vertex is included independently with probability
$$p = 1 - ir/N = 1 - rt. $$
(Note that this probability can be viewed as either a function of either $i$ or
$t$; we pass between these interpretations without comment.)
It follows from this assumption that the probability an edge $E \in \mc{H} $
is in $ \mc{H}(i) $ should be about $p^r$, and therefore we ought to have
\begin{equation}
\label{eq:Qloose}
Q(i) \approx |\mc{H}| p^r = ND p^r/r.
\end{equation}
Furthermore,
if a vertex $v$ is not saturated by
$ \mc{M}(i) $ then we should have
\begin{equation}
\label{eq:dloose}
d_v(i) \approx D p^{r-1}.
\end{equation}
Our main result (see Theorem~\ref{thm:main} below) is that
estimates (\ref{eq:Qloose}) and (\ref{eq:dloose})
hold for most of the evolution of the process.  This is
a generalization of a result of Bohman, Frieze
and Lubetzky \cite{bfl}, who proved an analogous
result for the special case
of $ \mc{H}_{2,3}$.

%of the process.

%\begin{theorem}

%\label{thm:main}

%With high probability we have

%\begin{align*}

%Q(i)  &= \frac{ND p^r }{r} (1+o(1)) \ \ \   \text{ and } \\

%d_v(i) &= Dp^{r-1}( 1+o(1)) \ \ \ \text{ for all } v \in V(i)

%\end{align*}

%for

%$$ i \le  \frac{N}{r} \left( 1 - (L/D)^{ \frac{1}{2 (r-1)} + o(1)} \right). $$

%\end{theorem}

%\noindent

%A more precise statement, which includes explicit error bounds is

%given in the next section.

In order to discuss our main result in more detail, we define
the random variable $$X=X(\mc{H}) := 1 - M r/ N$$
where $M$ is the number of steps before the random greedy matching algorithm on
$ {\mathcal H}$ terminates.  In other words, $X$ is the proportion of vertices
left unsaturated by the matching produced by the random greedy algorithm.
The following bound is
a Corollary of Theorem~\ref{thm:main}.
\begin{theorem}
\label{thm:X}
Let $r \ge 2$ and 
$ {\mathcal H}$ be an $r$-uniform,
$D$-regular hypergraph on $N$ vertices.
If the maximum
degree $L$ of a pair of vertices in ${\mathcal H}$ satisfies
$L = D/ ( \log N)^{ \omega(1)} $ and $X(\mc{H})$ is the proportion of vertices that are
not saturated by the matching produced by the random greedy algorithm then
with high probability we have
\[ X( \mc{H}) \le \left( \frac{L}{D} \right)^{ \frac{1}{2 (r-1)} + o(1)}. \]
\end{theorem}
\noindent
Previous analyses of the random greedy matching algorithm
due to Spencer \cite{joel} and, independently,
R\"odl and Thoma \cite{rodl} showed that if $ L = o(D) $ then
we have $ X( \mc{H}) = o(1) $ with high probability.  Note that this
result applied to the
hypergraph $ \mc{H}_{\ell,k} $ gives an alternate proof
of the Erd\H{o}s--Hanani conjecture.
Wormald \cite{nick2} applied the differential
equations method for random graph processes to show that if $ \mc{H}$ is
an $r$-uniform, $D$-regular hypergraph on $N$ vertices such that
$ D = o(N)$ but
$D \to \infty$ sufficiently quickly as $N \to \infty$ then
$ X(\mc{H}) < D^{- \frac{1}{ 9r(r-1) + 3} + o(1)} $ with high probability.

Theorem~\ref{thm:X} takes the analysis of random greedy matching
up to a natural
barrier.  To describe this barrier we
assume estimates (\ref{eq:Qloose}) and (\ref{eq:dloose}) hold.
For a fixed vertex $v$ let $L_v$ be the set of vertices $u$ such that the
degree of $ \{u,v\}$ in $ \mc{H}$ is $L$.  Note that $|L_v|$
can be as large as
$ (r-1)D/L $.  Now early in the process
(when $p=1/2$, say) the expected number of vertices in $L_v$ that
are not saturated by $ \mc{M}$ can be as large $ p D/L $ and thus
can have variation as large as $ \sqrt{ D/L } $, roughly speaking.  This
yields variations
in vertex degrees that are as large as $ \sqrt{ D/L } \cdot L = \sqrt{ DL} $.
If these early variations in vertex degree persist then at the point when
$ Dp^{r-1} = \sqrt{DL} $ these variations will be as large as the expected
degree itself.  So, if these variations indeed persist then when we
reach this point vertex degrees could be
zero even though the expected vertex
degree is large.  Note that this is point where Theorem~\ref{thm:main} no
longer holds.  One would expect that in order to prove better bounds one would have to
show that the variations in vertex degree {\em decrease} as the process evolves.

But where do we expect the random greedy matching algorithm to
finally terminate?  If we assume that
estimates (\ref{eq:Qloose}) and (\ref{eq:dloose}) hold all the way
to termination then
when $ ND p^r = Np $ the number of unsaturated
vertices should be roughly the
same as the number of remaining edges.  At this stage a positive
proportion of the unsaturated vertices should be in no remaining edges;
these vertices would remain unsaturated to termination.  Thus, it is
natural to guess that random greedy matching terminates when the proportion of
unsaturated vertices is
roughly $ D^{-1/(r-1)} $.  (We note in passing that this line of reasoning
is suspect if $ L > D^{1 - \frac{1}{r-1}} $.  In this case, one suspects that
we will reach a point where degrees of pairs of vertices in $\mc{H}(i)$ are larger
than degrees of individual vertices before the supposed termination point.)  In the
context of random greedy packing, this line of reasoning leads to the
following conjecture.
\begin{conjecture}[folklore] Let $ 1<\ell<k $ be fixed.  With high probability \label{conj:gen}
$$ X( \mc{H}_{\ell,k}) = n^{ {-\frac{k-\ell}{{k \choose \ell}-1}} +o(1)}.$$
\end{conjecture}
\noindent
The $\ell=2, k=3$ case of this conjecture was recently proved by Bohman, Frieze
and Lubetzky \cite{bfl2} who
establish estimates for vertex degrees in
$ \mc{H}_{2,3}(i) $ with error bounds that decrease as the process evolves.  These
self-correcting estimates are proved using the critical interval method that is featured
in this paper and was introduced in \cite{bfl}.  It should be noted that
the sharp result given in \cite{bfl2} 
requires a large, carefully selected 
ensemble of random variables.

The related problem of proving the existence of a large matching
in an $r$-uniform, $D$-regular hypergraph $ \mc{H} $ has been widely studied
(see \cite{pip} \cite{aks} \cite{kr}).  The best known results are
due to Vu \cite{van} who used a semi-random (i.e. R\"odl nibble type)
method to show that
%if $ \mc{H}$ is and $r$-uniform, $D$-regular hypergraph
%on $N$ vertices where $ D \to \infty$ as $N \to \infty$ then
there
exists a matching in $ \mc{H}$ that saturates all but at most
$$ \left( \frac{L}{D} \right)^{ \frac{1}{ r-1}+ o(1)} $$
vertices where $L$ is the maximum degree of pairs of
vertices in $\mc{H}$.  Vu obtained stronger results when one adds
degree assumptions for larger sets of vertices.
%\begin{gather*}

%1 - \frac{S(n,r,k) \binom{k}{r}}{ \binom{n}{r}} \leq

%n^{{-\frac{1}{{k \choose r}-1} +o(1)}}, \text{ and } \\

%%k \geq r+3 \ \ \ \Rightarrow \ \ \

%1 - \frac{S(n,r,k) \binom{k}{r}}{ \binom{n}{r}} \leq

%n^{{-\frac{k-r}{3 \left({k \choose r}-1 \right)} +o(1)}}.

%\end{gather*}

The remainder of this paper is organized as follows.  In the next Section we
give a precise statement of our dynamic concentration result.  The
proof follows in Section~\ref{sec:proof}.  This proof uses the critical interval
method introduced by Bohman, Frieze and Lubetzky in \cite{bfl}, where they
prove Theorem~\ref{thm:X} for
the special case $ \mc{H}_{2,3}$.  In this note we show that
the techniques introduced in \cite{bfl} are robust enough to 
handle the general case (with the introduction of some delicate
calculations necessitated by the large pairwise degrees).

\section{Dynamic Concentration}

Throughout this section we assume that $ {\mathcal H}$ is an $r$-uniform,
$D$-regular hypergraph on $N$ vertices where $r$ is a fixed constant and
$ D \to \infty$ as $N \to \infty$.  We also assume that the maximum
degree $L$ of a pair of vertices in ${\mathcal H}$ satisfies
$L = o(D/ \log^5 N)$ .
%\pbt{$L = D/ \log^{\omega(1)} N $}. \pbc{This still said before I changed it.}

In order to make the estimates (\ref{eq:Qloose}) and
(\ref{eq:dloose}) precise
we introduce error bounds for $Q$ and $ d_v $.
Define
\begin{align*}
e_q & = 90r^2 N L p^{2-r} \log N  \left( 1 - r \log p \right)^2 \\
e_d & = \sqrt{ 6r LD \log N} \left(1 - r \log p \right)
\end{align*}
Further define the stopping time $T$ to be
the first step $i$ such that
\begin{align*}
\left|Q(i) - \frac{ND}{r} p^r \right| & > e_q, \text{ or} \\
|d_v(i) - Dp^{r-1}| & > e_d \text{ for some } v \in V(i)
\end{align*}
\begin{theorem}
\label{thm:main}
With high probability we have
$$ N - Tr = O \left( N \cdot \left( \frac{L}{D} \right)^{ \frac{1}{ 2(r-1)}} \log^{ \frac{5}{ 2(r-1)}} N \right). $$
\end{theorem}

\section{Proof}

\label{sec:proof}

We begin with a brief overview of the critical interval method, which is a refinement of the
differential equations method for proving dynamic concentration.  In a standard application of the
differential equations method, we have a sequence of random variables $  Z(0), Z(1), \dots $ that is
determined by some combinatorial random process on $n$ points, and our
%Let $ \mc{F}_i $ be the natural
%filtration of the underlying probability space.  If our sequence of random variables satisfies 
%a condition of the form 
%\[ E \left[ Z(i+1) - Z(i) \mid \mc{F}_i \right] \approx g\left( \frac{Z(i)}{n} , \frac{i}{n} \right) \]
%where $g$ is a well-behaved (e.g. smooth) function on $ \mb{R}^2$ then we might be able to 
dynamic concentration statement is
\begin{equation}
\label{eq:idealized}
 Z(i)= z(i/n) \pm e_z(i/n) \ \ \ \ \text{ for } \ \ \ \ i = 0, 1, \dots, M(n) 
\end{equation}
with high probability.  Note that we use the symbol ``$\pm$" in two distinct ways: sometimes we write $a= b \pm c$ meaning that $a$ is in the interval $[b-c, b+c]$ whereas other times we simply use ``$\pm$" as a symbol that could either be ``$+$" or ``$-$." The meaning should be clear from context.  The deterministic trajectory function $ z$ is usually determined by the one-step
expected changes in $ Z(i)$ and the initial condition $ z(0) = Z(0)$. 
The error function $ e_z $ is a carefully chosen, slowly growing, function.  It is often convenient to introduce a continuous time variable $t$ that we relate to the
steps of the process by setting $ t = t(i) = i/n $.  This allows us to view the
function $ z(t)$ as a scaling limit for the sequence $ Z(i) $.  
%(We note in passing
%that this discussion assumes that both time and the random variable itself scale linearly with $n$.  But other scalings are possible.  Indeed, 
%the random variables considered in this
%paper do not generally scale linearly while time does scale linearly.)  

In a standard application of the differential equations method we prove the
dynamic concentration statement (\ref{eq:idealized}) by two applications of a martingale deviation
inequality.  We introduce 
a stopping time $T$, which is defined to be the minimum of $M = M(n)$ and the first step $ i $  at which (\ref{eq:idealized}) fails.  We then
define the two sequences of random variables $\mc{D}Z^{+}(i), \mc{D}Z^{-}(i)$ as follows:
\[ \mc{D}Z^{\pm}(i) = Z(i \wedge T) - z( t \wedge (T/n)) \pm e_z( t \wedge (T/n )). \]
Note that violation of the upper bound in (\ref{eq:idealized}) is equivalent to $ \mc{D}Z^-(T) = \mc{D}Z^-(M) > 0 $ and violation of
the lower bound in (\ref{eq:idealized}) is equivalent to $ \mc{D}Z^+(T) = \mc{D}Z^+(M) < 0 $.  Note further that $ \mc{D}Z^-(0) = -e_z(0) $
and $ \mc{D}Z^+(0) = e_z(0) $.  If $ \mc{D}Z^- $ is a supermartingale and $ \mc{D}Z^+ $ is a submartingale, then violation of
(\ref{eq:idealized}) is contained in the event that one of these martingales has a large deviation.
We choose the error functon $e_z(t)$ so that $ \mc{D}Z^- $ is a supermartingale and $ \mc{D}Z^+ $ is
a submartingale and $ e_z(0) $ is sufficiently large to make the probabilities of these martingale
deviations small.  We emphasize that the introduction of this stopping time $T$ is an 
important detail in the proof as it allows us to assume the bounds in (\ref{eq:idealized}) when 
we establish the martingale condition and apply
the martingale inequality.

Our proof of Theorem~\ref{thm:main} requires even greater control over the random variable $Z$ when
we are establishing the martingale condition.  This is what the critical interval method provides.
For each variable $Z$ treated by Theorem~\ref{thm:main} and each bound (i.e. upper and lower) we 
introduce a {\em critical interval} $ I_Z(t) = \left[a_Z,b_Z \right] $ which has one end at the bound we
are trying to establish and the other end slightly closer to the trajectory $z(t)$.
The upper critical interval is
$$ I_Z(t) = \left[ z(t) + e_z(t) - f_z(t), z(t) + e_z(t) \right] $$
where 
%$s_Z$ is the scaling of the random variable $Z$ and 
the width $ f_z(t) $ will be chosen below.
%(The scaling for $ d_v$ is simply $D$ and the scaling for $Q$ is $ ND/r$ .)
Simillarly, the lower critical interval is
$$ I_Z(t) = \left[ z(t) - e_z(t) , z(t) - e_z(t) + f_z(t) \right] $$
We can view violation of the dynamic concentration statement given
by Theorem~\ref{thm:main} as the event that some variable manages to cross one of its critical intervals.
In order to bound the probability of this event we consider a large collection of martingales.  We have one such
martingale for each variable, each bound (upper and lower), and each step of the process that the
random variable in question might enter the critical interval for the last time before crossing the interval.

Consider a random variable $Z$ in the collection of random variables treated by Theorem~\ref{thm:main}, some step $j$ of the process, and
the upper bound on $Z$.  We introduce a stopping time that is specialized to the event that variable $Z$ enters
its upper critical interval at step $j$ and proceeds to cross the interval without leaving it.
Define $ T_{Z,j}$ to
be the minimum of the global stopping time $T$ (which is defined in
Section~2 above) and the first step $i \ge j$ when $ Z(i)$ is not in its upper critical interval.  We simply
have $ T_{Z,j}= j $ if $Z(j)$ is not in the upper critical interval.  We consider the sequence of random
variables 
$$ \mc{D}Z_j^-(i) = Z( i \wedge T_{Z,j}) - z(t \wedge ( T_{Z,j}/N ))  - e_z( t \wedge ( T_{Z,j}/N )) \ \ \  \text{ for } \  i=j, \dots $$
Now, assuming that we have a suitable bound on the one step
changes in each variable $Z$, the event $ T = i$ and $ Z(i) > z(t) + e_z(t) $ is contained in the event that there exists 
a $ j<i $ such that $ \mc{D} Z_j^- (j) \approx - f_z(j/N) $ and $\mc{D} Z_j^-(i) > 0 $.  If $ \mc{D} Z_j^-$ is a supermartingale then each such event
is the event that this martingale has a large deviation.  We establish bounds on these events that are small enough that a
simple application of the union bound -- taking the union over all variables, bounds and starting points $j$ -- shows that the probability that of any
 event in the collection occuring is small.  Theorem~\ref{thm:main} follows.

We stress that the introduction of the stopping time $ T_{Z,j} $ allows us to assume that $ Z$ is in the
critical interval when we are establishing the martingale condition for $Z$. (Of course the other
random variables are not so constrained.)
The reason that we focus our attention on these critical intervals is the fact that the expected one-step
changes in the variables we consider have self-correcting terms.  These terms introduce a drift back toward
the expected trajectory when $Z$ is far from the expected trajectory.  By restricting our attention to the critical
intervals we make full use of these terms.  See \cite{nick} and
\cite{mike} for early applications of
this self-correcting phenomenon in applications of the differential equations method for
proving dynamic concentration.  As we noted above, the critical interval method we use here was introduced in
\cite{bfl}.

We close this preamble with some notation conventions and a lemma
that we use below.  For an arbitrary random variable $Z$ we define
\[ \Delta Z(i) = Z(i+1) - Z(i). \]
We let $ \mc{F}_i $ be the filtration of the probability space
given by the first $i$ edges chosen by the random greedy matching process.
\begin{lemma}
\label{lem:pat}
Suppose $ (x_i)_{i\in I} $ and $ (y_i)_{i \in I}$ are real numbers such that
$|x _i - x| \le \delta$ and $|y_i-y| < \epsilon$
for all $i \in I$.  Then we have
\[ \left| \sum_{i \in I} x_i y_i -  \frac{1}{|I|} \left( \sum_{i \in I} x_i \right)
\left( \sum_{i \in I} y_i \right) \right|  \le  2 |I| \delta \epsilon \]
\end{lemma}
\begin{proof}
%We consider the function $$z= \displaystyle \sum_{i\in I} x_i y_i$$ and suppose we know that

%$$x_i \in x \pm \delta, y_i \in y \pm \epsilon$$ for some $x > \delta>0$ ,and  $y> \epsilon > 0$.

The triangle inequality gives
$$\displaystyle \left| \sum_{i \in I} (x_i - x)(y_i - y) \right| \leq |I|\delta\epsilon.$$
Rearranging this inequality gives
\begin{equation*}
\begin{split}
\sum_{i \in I} x_i y_i  & =  x\sum_{i \in I}y_i + y \sum_{i \in I} x_i
- |I|xy \pm |I|\delta \epsilon \\
& = \frac{1}{|I|} \left( \sum_{i \in I} x_i \right)
\left( \sum_{i \in I} y_i \right)  - |I| \left(\frac{1}{ |I|} \sum_{i \in I} x_i  - x\right)
\left(\frac{1}{ |I|} \sum_{i \in I} y_i  - y\right) \pm |I|\delta \epsilon.
\end{split}
\end{equation*}
\end{proof}

%Note that imposing the constraints $x_i=y_i$ in the special case $x=y, \delta=\epsilon$  still gives a good upper bound on the sum of squares $\sum_{i \in I} x_i^2$.

%However in that case, the lower bound is not as good as Cauchy-Shwartz, which gives

%$$\displaystyle z \geq \frac{\left(\sum_{i \in I} x_i \right)^2}{|I|}$$

\subsection{Vertex degrees}

Let $v$ be a fixed vertex.  As usual in applications of the
differential equations method for establishing dynamic concentration,
we begin with the expected one-step change in $ d_v $ (i.e. we begin
with the trend hypothesis).  We have
\begin{equation}
E\left[ \Delta d_v(i)| \mathcal{F}_i \right] = \displaystyle
- \frac{1}{Q} \sum_{E \in \mc{H}(i): v \in E } \sum_{u \in E\setminus\{ v\}} d_u(i) \pm
d_v(i) \binom{r}{2} \frac{L}{Q},   
\label{eq:trendd}
\end{equation}
where $ {\mathcal F}_i $ is the filtration defined by the 
random greedy matching process.  We note that (\ref{eq:trendd}) 
does not take into account the contribution to the expected 
change in $ d_v$ that comes from the selection of an edge 
that contains $v$ itself.  Of course, this event causes a rather
dramatic change in $d_v$, which could complicate our analysis.  
Furthermore, we are no 
longer interested in $d_v$ after $v$ leaves $V(i)$.  
This is handled formally by 
setting $ d_v(i+1) = d_v(i) $ if $ v \not\in V(i+1) $, and 
(\ref{eq:trendd}) takes this convention into account.

We begin with the upper bound on $ d_v$.  Our critical interval is
$$ [ Dp^{r-1} + e_d - f_d, Dp^{r-1} + e_d], $$  
where
%\begin{align*}
\[ f_d =  \sqrt{ 6r LD \log N} \ \ \  \ \ \ \text{ and } \ \ \ \ \ \ 
e_d  = f_d  \left(1 - r \log p \right). \]
%\end{align*}
Note that $f_d$ does not change in time and 
that $ e_d$ is increasing.
%The 
%function $e_d$ is define above
%and the function $ f_d$ will be determined below.  

For each step $j$ of the process
we define the sequence of random variables
\begin{displaymath}
\mc{D} d^+_{v,j}(i) := d_v(i) - Dp^{r -1} - e_d  \ \ \ \ \ \text{ for } i \ge j
\end{displaymath}
with the stopping time $ T_{v,j} $ defined to be the
minimum of $T$ and the smallest index $i \geq j$ such that $d_v(i)$ is
not in the critical interval or $ v \not\in V(i)$.  Note that if $ d_v(j) $ is not in the
critical interval then we simply have $ T_{v,j}=j$.
We prove dynamic concentration by considering
the sequence of random variables $\mc{D}d^+_{v,j}(j), \ldots, \mc{D} d^+_{v,j}(T_{v,j})$.
We chose $f_d$ and $e_d$ (with foresight) so 
that this sequence is a supermartingale
with respect to the natural filtration $\mathcal{F}_i$. For $j \leq i < T_{v,j}$ we have
\begin{equation*}
\begin{split}
E\left[ \Delta  \mc{D}d^+_{v,j}| \mathcal{F}_i \right] &
\leq - \frac{1}{Q} \sum_{E \in \mc{H}(i): v\in E} \sum_{u \in E \setminus \{v\}} d_u(i)
+ \frac{D r (r-1)}{N} p^{r - 2} - \frac{1}{N} e_d '\\
& \hskip2cm + O\left(\frac{L d_v}{Q}+\frac{D}{N^2} p^{r-3} + \frac{1}{N^2}e_d ''\right)\\
& \leq -\frac{ \left(Dp^{r-1} + e_d - f_d \right) \left( r-1 \right) \left(Dp^{r-1} - e_d \right) }{ NDp^r/r +e_q} + \frac{D r (r-1)}{N} p^{r - 2} \\
& \hskip2cm  - \frac{1}{N} e_d ' + O\left(\frac{L d_v}{Q}+\frac{D}{N^2} p^{r-3} + \frac{1}{N^2}e_d ''\right)\\
& \leq \frac{r(r-1)}{Np} f_d - \frac{1}{N} e_d ' \\
& \hskip2cm + O\left(\frac{(e_d - f_d) e_d}{NDp^r} + \frac{e_q}{N^2p^2} +\frac{L d_v}{Q}+\frac{D}{N^2} p^{r-3} + \frac{1}{N^2}e_d ''\right)
%\\
%& \leq \frac{r(r-1)}{Np} f_d - \frac{1}{N} e_d ' \\
%& \hskip2cm + \frac{ e_d}{ pN} \cdot O\left(\frac{e_d p^{1-r}}{D} + \frac{e_q}{e_dNp}
%+\frac{L}{e_d}+\frac{D}{N e_d } p^{r-2} + \frac{1}{p N \log N}\right)\\
\end{split}
\end{equation*}
Note that we use the assumption that $ d_v(i) $ lies in the critical interval.
Also note that in order to get the desired supermartingale condition it is necessary
to choose $e_d$ and $f_d$ so that
\begin{equation}
\label{eq:var1}
e_d' >  \frac{r(r-1)}{p} f_d.
\end{equation}
(Of course, this equation plays a central in our choice of the functions $f_d$ and $e_d$.)
%Note that we have the supermartingale condition as (\ref{eq:var1}) 
%is satisfied by $e_d$ and $f_d$, and, 

For the given error functions $e_d, e_q$, we have
\begin{align}
& \frac{(e_d - f_d) e_d}{NDp^r} + \frac{e_q}{N^2p^2} +\frac{L d_v}{Q}+\frac{D}{N^2} p^{r-3} + \frac{1}{N^2}e_d '' \nonumber  \\
& \hskip1cm
\le \frac{ e_d}{ N p} \cdot O\left(\frac{e_d p^{1-r}}{D} + \frac{e_q}{e_dNp} +\frac{L}{e_d}+\frac{D}{N e_d } p^{r-2} + \frac{1}{N p }\right) \label{eq:errors} \\
& \hskip1cm
\le \frac{ e_d}{ N p} \cdot O\left(\frac{ \sqrt{L}( \log N)^{3/2} p^{1-r}}{\sqrt{D}} \right)
+ \frac{ e_d}{ N p} \cdot o\left( \frac{\sqrt{L}}{\sqrt{D}} \nonumber
+ \frac{ \sqrt{D}}{N \sqrt{L} } + \frac{1}{ \sqrt{N}} \right).
\end{align}
(We note that these estimates make repeated use of the simple inequality $ D < NL$.)
By assuming that $p$ is a sufficiently large constant times
\[ \left(\frac{L}{D} \right)^{\frac{1}{2(r-1)}} \log^{\frac{5}{2(r-1)}} N \]
we see that the expression in (\ref{eq:errors}) can be made smaller than any constant
times $ e_d / ( Np \log N) $.  As the error functions $f_d$ and $e_d$ 
satisfy (\ref{eq:var1}), the supermartingale condition is satisfied.

We use a supermartingale inequality to bound the probability that the random variable $ \mc{D} d^+_{v,j}(T_{v,j})$ is positive. We use the following Lemma 
(see \cite{r3t} for a proof).
\begin{lemma} Let $X(i)$ be a supermartingale, such that  $-\Theta \leq \Delta X(i) \leq \theta$ for all $i$, where $\theta < \frac{\Theta}{10}$. Then for any $a < \theta m$ we have $$Pr(X(m) -X(0) > a) \leq \exp \left(- \frac{a^2}{3 \theta \Theta m}\right).$$ \end{lemma}
\noindent
Since $d_v$ is non-increasing, $Dp^{r -1}$ is decreasing and $e_d$ is increasing, the one step change in $ \mc{D} d^+_{v,j}$ is bounded above by the one step change in $Dp^{r -1}$, which is at most $$\theta=\displaystyle \frac{D(r-1)}{N}  (1+o(1)).$$ For a lower bound on $\Delta d^+_{v,j}$, note that the one step change in $e_d$ is negligible compared to the maximum possible one step change in $d_v$, which occurs when we pick an edge containing a vertex that has pairwise-degree $L$ with $v$.  So we can set $\Theta=r L (1+o(1))$.

Now, if $d_v$ crosses the upper boundary of its critical interval at the stopping time $T$, then there is some step $j$ (with $T=T_{v,j}$) such that $$ \mc{D} d^+_{v,j}(j) \leq -f_d(t(j)) + \frac{D(r-1)}{N}  (1+o(1))$$
and $d^+_{v,j}( T_{v,j})>0$.  Applying the lemma (and noting $ D/N = o(f_d) $)
we see that the probability
of the supermartingale $d^+_{v,j}$ having such a large upward deviation has probability at most
$$\exp\left\{ -\frac{f_d^2}{3 \frac{D(r-1)}{N} (r L) (\frac{Np}{r})} (1+o(1)) \right\}.$$
As there are $O(N^2)$ such supermartingales, we would like the above expression to be $o(N^{-2})$.
Thus, it suffices to take
$$f_d = \sqrt{6 r LD \log N}.$$
Furthermore this choice also satisfies (\ref{eq:var1}).  (Note that, in fact, this condition
together with (\ref{eq:var1}) essentially determines the error functions $e_d$.)

%

%Note that this choice together with the condition $ e_d' > (r^2 f_d)/p $ used in

%determining the martingale calculation above determines $e_d$.

%

%

%$$f_d = \sqrt{ LD \log N} 6^\frac{1}{2} r^\frac{3}{2} \log p$$

%

%So long as we restrict $$p= \Omega\left( \left(\frac{L}{D} \right)^{\frac{1}{2(r-1)}} \log^{\frac{5}{2(r-1)}} N \right)$$ these choices make the sequence of variables $d^+(v,j) \ldots d^+(v,T_j)$ a supermartingale whose probability of having a large upward drift is $o(N^{-2})$.

Thus, the probability that $T$ is less than bound stated in Theorem~\ref{thm:main} due to
a violation of the upper bound on $d_v$ goes to zero as $N$ tends to infinity.

The lower bound for $d_v$ is similar.

\subsection{ Number of edges}

We again begin with the trend hypothesis.  We have
$$E\left[ \Delta Q(i)| \mathcal{F}_i \right] =  \displaystyle
-\frac{1}{Q} \sum_{A \in \mc{H}(i)} \sum_{v \in A} d_v(A)   + O(L)
=
\displaystyle
- \frac{1}{Q} \sum_{v \in V(i)} d_v^2(i)  +O(L)   $$
For $i < T$ we have
$$ \displaystyle  \sum_{v \in V(i)} d_v^2 =  \frac{ (rQ)^2}{ Np} \pm 2Npe_d^2, $$
by an application of Lemma~\ref{lem:pat}, and therefore
\[ E\left[ \Delta Q(i)| \mathcal{F}_i \right] =  - \frac{ r^2 Q}{Np} \pm \frac{ 2 Np e_d^2}{ Q} + O(L). \]

We work with the upper bound on $ Q(i) $.    Our critical interval is
$$ \left[ \frac{ND}{r} p^{r} + e_q - f_q,  \frac{ND}{r}p^{r} + e_q\right], $$
where
\[  f_q = 6r^2 N L \log N p^{2-r} \ \ \ \ \ \text{ and } \ \ \ \ \ 
e_q = 
15 f_q \left( 1 - r \log p \right)^2. \] 
Note that both $f_q$ and $e_q$ are non-decreasing in time.
For each step $j$ of the process
we define the sequence of random variables
$$ \mc{D} Q^+_j(i) := Q(i) - \frac{ND}{r} p^r - e_q $$
with the stopping time $ T_{j} $ defined to be the
minimum of $T$ and the smallest index $i \geq j$ such that $Q(i)$ is
not in the critical interval.
%Note that if $ d_v(j) $ is not in the

%critical interval then we simply have $ T_j=j$.

%We prove dynamic concentration by considering

%the sequence of random variables $d^+_{v,j}(j), \ldots, d^+_{v,j}(T_j)$.

%We chose $e_d$ (with foresight) so that this sequences is a supermartingale

%with respect to the natural filtration $\mathcal{F}_i$.

We begin by showing that $ \mc{D} Q^+_j(j), \dots, \mc{D} Q^+_j(T_j) $ is a supermartingale.
For $j \leq i < T_j$ we have
\begin{equation*}
\begin{split}
E\left[ \Delta \mc{D} Q_j^+(i)| \mathcal{F}_i \right]
& \leq \displaystyle -  \frac{ r^2 Q}{Np}  + r D p^{r-1} - \frac{1}{N} e_q '
+  \frac{ 2Np e_d^2}{ Q} + O\left( L + \frac{D}{N}p^{r-2} + \frac{1}{N^2} e_q '' \right)\\
%& \leq \displaystyle - \frac{1}{Q}  \left( \frac{r^2 Q^2}{Np} - Npf_d^2 \right)
%+ r D p^{r-1} - \frac{1}{N} f_q '\\
%& \hskip2cm + O\left(L + \frac{D}{N}p^{r-2} + \frac{1}{N^2} f_q '' \right)\\
& \leq \displaystyle  - \frac{ r^2 (e_q - f_q) }{Np} -\frac{1}{N} e_q ' +  \frac{ (2r + o(1)) p^{1-r} e_d^2}{ D } \\
& \hskip3.5cm  +
O\left(  L + \frac{D}{N}p^{r-2} + \frac{1}{N^2} e_q '' \right) \\
%& \hskip2cm + O\left(L + \frac{D}{N}p^{r-2} + \frac{1}{N^2} f_q '' \right)\\
%& \leq \displaystyle \frac{r}{D} p^{1-r}f_d ^2 - \frac{r^2}{Np}g_q - \frac{1}{N}f_q ' \\
%& + O\left(\frac{g_q ^2}{N^2 D p^{r +1}} + L + \frac{D}{N}p^{r-2} + \frac{1}{N^2} f_q '' \right)\\
\end{split}
\end{equation*}
In order to get the supermartingale condition it suffices, up to constant factors, to take $$e_q > e_d^2 N p^{2-r}/ D .$$  Note that this determines
the main terms in the choice of $ e_q$ above.  As
$f_q = 6r^2NL \log N p^{2-r}, $
we have
\[
- \frac{ r^2 (e_q - f_q) }{Np} +  \frac{ (2r + o(1)) p^{1-r} e_d^2}{ D }
\le -L p^{1-r} (\log N) (1 - r \log p)^2. \]
This clearly dominates the remaining error terms (note that $ e_q'>0$) and therefore
the sequence $\mc{D} Q^+(j), \ldots, \mc{D} Q^+(T_j)$ is a supermartingale.

Now we apply the Hoeffding-Azuma inequality to bound the
probability that the random variable $ \mc{D} Q^+(T_j)$ is positive.
The lemma we use is as follows:

\begin{lemma}Let $X_j$ be a supermartingale, with $|\Delta X_i| \leq c_i$ for all $i$. Then
$$P(X_m - X_0 \geq a) \leq \displaystyle \exp\left(-\frac{a^2}{2 \displaystyle \sum_{i\leq m} c_i^2 }\right).$$ \end{lemma}

Since $i<T$ implies bounds on degrees, we have
$$|\Delta \mc{D} Q^+| \le (1 + o(1))r e_d \le \sqrt{7r^3 LD \log N}( 1-r\log p).$$
Thus, if $Q$ crosses its upper boundary at the stopping time $T$, then there is some
step $j$ (with $T=T_j$) such that $$ \mc{D} Q^+ (j)\leq - f_q(t(j)) + O(\sqrt{LD} \log^{3/2} N)$$ and
$ \mc{D} Q^+(T_j)>0$. Applying the Hoeffding-Azuma we see that the probability of the
supermartingale $ \mc{D} Q^+$ having such a large upward deviation has probability at most
\begin{multline*}
\exp\left\{ -\frac{[(1+o(1))6r^2NL \log N p^{2-r} ]^2}{2(Np)
[ 7r^3 LD \log N(1-r\log p)^2]}  \right\} \\
\leq \exp\left\{ -(1+o(1)) \frac{18r}{7} \cdot \frac{NL}{D}  \cdot
 \frac{ p^{3-2r}}{ (1-r\log p)^2} \cdot \log N  \right\} = o(N^{-1})
\end{multline*}
where $p = p(j)$. Note that we have used $D < NL$ again and that the constants have been chosen to deal
with $p$ constant. As there are at most $O(N)$ such supermartingales, the probability
that $T$ is less than the bound stated in Theorem~\ref{thm:main} due to $Q(i)$ breaching
the upper bound tends to zero as $N$ tends to infinity.

The lower bound for $Q$ is similar.

\end{document}